\documentclass[11pt]{article}

\usepackage{amsmath,amsthm,verbatim,amssymb,amsfonts,amscd, graphicx}
\usepackage[square,sort,comma,numbers]{natbib}
\usepackage{graphicx}
\usepackage{mathrsfs}
\usepackage{enumerate}
\usepackage{mathtools}
\usepackage{lipsum}         
\usepackage{xargs}                      
\usepackage{authblk}
\usepackage{titling}

\theoremstyle{plain}
\newtheorem{theorem}{Theorem}

\newtheorem{lemma}{Lemma}

\theoremstyle{definition}

\newcommand\blfootnote[1]{%
  \begingroup
  \renewcommand\thefootnote{}\footnote{#1}%
  \addtocounter{footnote}{-1}%
  \endgroup
}

\begin{document}

\title{Photo-acoustic tomography in the rotating setting}
\author{Guillaume Bal\thanks{Department of Statistics, University of Chicago,} \hspace{1mm} and
Adrian Kirkeby\thanks{Department of Applied Mathematics and Computer Science, Technical University of Denmark}\hspace{0.5mm} \thanks{Department of Mathematical Sciences, Norwegian University of Technology and Science} }
\thanksmarkseries{arabic}
\maketitle
\blfootnote{E-mail: guillaumebal@uchicago.edu  and   adrki@dtu.dk}
\begin{abstract}
    Photo-acoustic tomography is a coupled-physics (hybrid) medical imaging modality that aims to reconstruct optical parameters in biological tissues from ultrasound measurements. As propagating light gets partially absorbed, the resulting thermal expansion generates minute ultrasonic signals (the photo-acoustic effect) that are measured at the boundary of a domain of interest. Standard inversion procedures first reconstruct the source of radiation by an inverse ultrasound (boundary) problem and second describe the optical parameters from internal information obtained in the first step. 
    
    This paper considers the rotating experimental setting. Light emission and ultrasound measurements are fixed on a rotating gantry, resulting in a rotation-dependent source of ultrasound. The two-step procedure we just mentioned does not apply. Instead, we propose an inversion that directly aims to reconstruct the optical parameters quantitatively. The mapping from the unknown (absorption and diffusion) coefficients to the ultrasound measurement via the unknown ultrasound source is modeled as a composition of a pseudo-differential operator and a Fourier integral operator. We show that for appropriate choices of optical illuminations, the above composition is an elliptic Fourier integral operator. Under the assumption that the coefficients are unknown on a sufficiently small domain, we derive from this a (global) injectivity result (measurements uniquely characterize our coefficients) combined with an optimal stability estimate. The latter is the same as that obtained in the standard (non-rotating experimental) setting.
    
\end{abstract}
\section{Introduction}
\label{intro}

Photo-acoustic tomography (PAT) is a coupled-physics (also known as hybrid) imaging method that aims to reconstruct the optical parameters of biological tissues. The optical parameters are known to provide valuable, high-contrast information about for example cancerous tissue, and hence are of clinical interest \cite{BSFECO-JBO-09,wangPAT,patcancer,bal2010inverse}. To image these parameters, PAT leverages the thermoelastic expansion generated by absorbed light. A domain of interest is illuminated by electromagnetic waves and the absorption of light within the sample causes a rapid expansion of the material proportional to the amount of absorbed photons; this conversion of light to ultrasound is what is known as the \emph{photo-acoustic effect}. The expansion initiates a pressure wave, and this wave is recorded by acoustic detectors at the boundary of the sample. The objective of PAT is to reconstruct the optical parameters from such acoustic measurements.\\
\newline
A model approximating the propagation of light in the diffusive regime is the following: 
\begin{equation}
    \begin{split}
        -\nabla\cdot D \nabla u + \sigma u = 0  \quad \text{in } \Omega, \\
        u = g \quad \text{on } \partial \Omega.
    \end{split}
    \label{diffusion}
\end{equation}
Here $\Omega \in \mathbb{R}^d$ (typically $d=2,3$) is the object domain, $u(x)$ is the light/photon intensity, $D(x)$ is the diffusion coefficient and $\sigma(x)$ is the absorption parameter. The illumination is modelled as the boundary condition $g$. The thermo-elastic expansion is proportional to the amount of absorbed light, and is given by 
\begin{equation}\label{eq:H} H(x) = \mu(x) \sigma(x) u(x),
\end{equation} 
where $\mu(x)$ is a proportionality coefficient known as the Gr\"{u}neisen coefficient.
It is known that the reconstruction of the three parameters $(\sigma, D,\mu)$ is not possible without prior assumptions or multi-color measurements \cite{bal2011multi,BR-IP-12}. In the rest of the paper, we assume $\mu(x)$ known.

To model the acoustic wave resulting from the thermoelastic expansion, we consider the linear wave equation
\begin{equation}
    \begin{split}
        &(\partial_t^2 - c^2(x)\Delta) v(t,x) = 0, \quad \text{in }(0,\infty)\times\mathbb{R}^d,\\
        &v(0,x) = H(x), \quad x \in \Omega,\\
        &\partial_t v(0,x) = 0, \quad x \in \mathbb{R}^d, 
        \end{split}
        \label{wave1}
\end{equation}
where $v(t,x)$ is the acoustic pressure and $c(x)$ is the sound speed, assumed to be known. 
For some domain $\Omega_M$ such that $\Omega \subset \Omega_M$, one measures $v(t,x)$ for $(t,x) \in [0,T] \times \partial \Omega_M$ for some sufficiently large duration $T>0$. We refer to this as the full data. The quantitative reconstruction of $D$ and $\sigma$ is then often considered in two successive steps: 
\begin{enumerate}
    \item The reconstruction of $H(x)$ from $v(t,x)\big|_{[0,T]\times \partial \Omega_M}$. In the case of full data, this is a well-posed problem and several inversion methods exist\cite{timereversal2,kuchment,uhlmann,timereversal}. 
    \item The reconstruction of $D(x)$ and $\sigma(x)$ from $H(x)$ (assuming $\mu(x)$ known). This is quantitative PAT, which is known to be reasonably well-posed \cite{bal2011multi,bal2010inverse,gao2012quantitative,kuchment}. 
\end{enumerate}
The fact that both steps are well-posed is what makes PAT an attractive medical imaging modality. However, the requirements for the decoupling into two separate steps can be hard to meet in an experimental situation. Having access to a full measurement of the acoustic wave requires the illuminated object to be fully surrounded by acoustic measurement devices, a situation which is hard to achieve in most situations of interest. Since Step 1 cannot be stably completed without access to (sufficiently) full data, one instead relies on doing several partial measurements, changing the position and the illumination patterns for each measurement, for instance by rotating both the light source and ultrasound detectors \cite{BSFECO-JBO-09,LNWEOA-SPIE-15}. In \cite{rotmes}, such a situation, with $D$ constant, was considered, and it was shown that the absorption coefficient $\sigma$ can be stably determined in situations where the object (or measurement device) is rotating. In the current paper we investigate the situation were both $D$ and $\sigma$ are unknown, and the measurements and illuminations are rotating. We show that under certain assumptions on the parameters, this situation also allows for stable determination of both $D$ and $\sigma$. 

\medskip

Let us conclude this introductory section by mentioning that PAT is one of the many hybrid (coupled-physics) medical imaging modalities that have emerged in recent years. For a brief list of mathematical description of such methods, we refer the reader to, e.g., \cite{alberti,A-Sp-08,B-IO-12,MR3289684,MY-IP-04,PS-IP-07,WW-W-07} and their multiple references. The rest of the paper is organized as follows. The setting and main uniqueness and stability result are presented in section \ref{sec:main}. The details of the derivation are collected in section \ref{sec:ell} while some concluding remarks on the imaging modality and the mathematical tools are given in section \ref{sec:conclu}.

\section{Background and main results}
\label{sec:main}

\subsection{Sound propagation}
\label{SP}
 The linear acoustic wave equation is a reasonable description of sound propagation \cite{kuchment2011mathematics,uhlmann,rotmes}. The sound speed $c(x)$ is considered known in this paper.
Taking $v$ to be the solution of 
\begin{equation}
    \begin{split}
        &(\partial_t^2 - c^2(x)\Delta) v(t,x) = 0, \quad \text{in }(0,\infty)\times\mathbb{R}^d,\\
        &v(0,x) = f(x), \quad x \in \Omega,\\
        &\partial_t v(0,x) = 0, \quad x \in \mathbb{R}^d, 
        \end{split}
        \label{wave2}
\end{equation}
we define the acoustic measurement operator $\Lambda$ as 
\begin{equation*}
    \Lambda f = v(t,x) \big|_{[0,T]\times \partial \Omega_M}.
\end{equation*}

In the setting considered here, we do not have one, but rather a large number of ultrasound sources $f(x)=f_i(x)$ corresponding for instance to the rotating measurement setting, where a rotating set of light sources generates a different ultrasound source for each rotation. Since the ultrasound detectors also rotate in such a setting, $v(t,x)=\Lambda f$ is available only for $x$ in the support of the rotating detector, which is a small fraction of the boundary $\partial\Omega_M$. As a consequence, we cannot reconstruct the whole $f(x)$ stably from such measurements \cite{alberti}. All that we can reconstruct is the singularities of $f$ that are visible from the available measurements. This mapping from the visible singularities of $f$ to the available measurements is described by a standard microlocal procedure, which we now recall \cite{uhlmann}.

\medskip

Consider the Hamiltonian system:
\begin{equation}
    \begin{split}
    &\dot{X}(t) = \frac{c(X(t))}{2}\xi(t), \\
    &\dot{\xi}(t) = -\frac{\nabla c(X(t))}{2}|\xi(t)|, \\
    &X(0) = x_0, \quad \xi(0) = \xi_0. 
    \end{split}
    \label{hamilton}
\end{equation}
The curves $(X(t),\xi(t))$ are known as bi-characteristic curves, and each $X(t)$ is called a ray. We assume that $c(x) \in C^\infty(\mathbb{R}^d)$ is non-trapping, that is $|X(t)| \to \infty $ as $t \to \infty$, for all rays \cite{kuchment2011mathematics}.   An important feature of the rays is that they are curves along which the propagation of the singularities of a wave occurs, a fact that motivates the preceding assumptions on the measurement geometry, see \cite{uhlmann,rotmes}. To that end, we use the unit speed geodesics $\gamma_{x_0,\xi_0}(t)$, defined by the relation $\dot{\gamma}_{x_0,\xi_0}(t)= \dot{X}(t)/|\dot{X}(t)| $, where $X(t)$ is the solution to \eqref{hamilton} with initial conditions $(x_0,\xi_0)$.   \\
\newline 

We can then describe the measurement operator $\Lambda$ has the following Fourier integral operator(FIO) \cite{FIO,uhlmann}:
\begin{equation}\label{eq:vFIO}
     v(t,x) = \Lambda f (t,x) =(2 \pi)^{-d} \sum_{\tau = \pm}\int  \mathrm{e}^{i\phi_\tau(t,x,\xi)-i\xi \cdot y
    }a_\tau(t,x,\xi)\hat{f}(\xi) \mathrm{d}\xi + Rf(t,x)
\end{equation}
for $(t,x) \in [0,T]\times\partial \Omega_M$, where $R$ is a linear operator with smooth Schwartz kernel.
The phase functions $\phi_\pm$ are solutions to the eikonal equations
$$\mp \partial_t \phi_\pm = c(x)|\nabla_x \phi_\pm |, \quad  \phi_{\pm}(0,x) = x\cdot \xi, $$ and homogeneous of order one in $\xi$. As usual, 
$$\hat{f}(\xi) = \int_{\mathbb{R}^d}\mathrm{e}^{-i\xi\cdot y}f(y) \mathrm{d}y.$$
Note that if $c=\text{const.}$, then $\phi_{\pm} = x\cdot \xi \pm c |\xi|t$, and the solution is exact. The function $a_\pm$ is called a classical amplitude of order 0, and satisfies a recursive transport equation; see (\cite{FIO}, page 128) for details.

We have assumed here that the FIO could be represented with a single (global) phase function $\phi_\tau$. This can always be done for sufficiently short times. For longer times, and to avoid the presence of caustics, the above operator should really be written as a composition of a finite number of such terms, or more generally as a globally defined FIO \cite{FIO,uhlmann}. To simplify notation, we represent our FIO, mapping all the singularities of $f$ (that is $f$ up to a smooth term) to the available measurements by \eqref{eq:vFIO}.

\subsection{Detector model}
\label{sec:dm}
Assume a detector supported on $\Gamma_d$, a closed and bounded hypersurface in $\mathbb{R}^d$. The $i$'th measurement consists of acquiring $v(t,x)$ for $t\in [0,T]$ and $x \in \Gamma_i$, where $\Gamma_i$ is a translation and rotation of $\Gamma_d$.

To obtain enough information on the optical coefficients, we make the following assumption on the measurement setting. We require overlapping measurements, i.e., that $\mu(\Gamma_{i}\cap \left(\cup_{j\neq i} \Gamma_j \right)) > 0$, for all $i$, where $\mu$ is the Lebesgue (surface) measure on $\Gamma_d$. For a collection of $M$ measurements, we set the full measurement surface $\Gamma = \cup_{i =1}^M \Gamma_i$. We then assume we have a domain $\Omega_M$ such that $\Omega \subset \Omega_M$ and $\Gamma = \partial \Omega_M$. In the case of rotating measurements, we obtain that $\Gamma = \partial \Omega_M$ for some ball $ \Omega_M = B_r$ of radius $r$ containing $\Omega$.  \\
\newline 
To describe our measurement setting, we follow the approach  of \cite{uhlmann} used to analyze the inversion of the wave equation with incomplete data. Our requirements on $\Gamma$ and $T$ are the same as those required to get stable inversions with incomplete data. Define 
\begin{equation*}
    \mathcal{G} = \{(t,x) : x \in \Gamma, 0 < t < s(x)\}, 
\end{equation*}
where $s(x)$ is a continuous function determining the temporal measurement interval. We introduce the function $\tau_{\pm}(x,\xi)$, defined as 
$$\tau_{\pm}(x,\xi) = \max\{t \geq 0 : \gamma_{x,\xi}(\pm t) \in \Omega_M \},$$
where $\gamma_{x,\xi}( t)$ are the rays in \eqref{hamilton}. 
In order to make sure all of the singularities of the function $H(x)$ in \eqref{wave1} (and hence of the optical coefficients via \eqref{eq:H}) are captured by the measurement, additional requirements on $\mathcal{G}$ (and hence $\Gamma$ and the measurement interval length $T$) are 
\begin{itemize}
    \item $\forall x \in \Omega, \quad \exists z \in  \Gamma \text{ such that dist}(x,z) < s(z)$. Here dist$(x,z)$ is the length of the geodesic connecting $x$ and $z$, with respect to $\mathrm{d}s^2 = c^{-2}(x)\mathrm{d}x^2$. The conditions says that all of the wave should reach $\Gamma$ during the measurement period. 
    \item $\forall (x,\xi) \in T^*\Omega\backslash 0, \quad (\tau_k(x,\xi), \gamma_{x,\xi}(\tau_k(x,\xi)) \in \mathcal{G},$ for $k=+$ or $k=-$.  \\
    Here $T^*\Omega =\Omega\times \mathbb{R}^d$ is the cotangent bundle of $\Omega$, where we consider $\xi\not=0$ only. The motivation for this condition is that every point in the wavefront set of $H$ should reach the measurement surface.  
\end{itemize} 

The above conditions are satisfied if for example $\Gamma$ is a sphere containing $\Omega$ and an interval such that $s(x) > \max_{x,z \in \bar{\Omega}_M} \text{dist}(x,z) /2$ when the wave speed is constant. These conditions are also satisfied for non-trapping speeds for sufficiently long measurement times \cite{uhlmann}.

\medskip

Last, for the measurement at $\Gamma_i$, let $\{\varphi_j\}_{j=1}^M$ be a partition of unity subordinate to $\{\Gamma_j\}_{j=1}^M$. Note that there is a $K_i \subset \Gamma_i \setminus \left(\cup_{j \neq i}\Gamma_j\right)$ such that $\varphi_i|_{K_i} = 1$, while $\varphi_j|_{K_i} = 0$ for $j \neq i$. Next, we take $\psi \in C^\infty_0([0,2T])$ such that $\psi(t) = 1$ for $t \in [\varepsilon,T]$, with $\varepsilon$ sufficiently small, and set $\chi_i(x,t) = \psi(t)\varphi_i(x) $.  We now define our rotating measurement $V_i$ at $\Gamma_i$ by 
\begin{equation}
    V_i = \chi_i \Lambda f. 
    \label{mes1}
\end{equation}


Note that $f\equiv H=\mu\sigma u$ above should really be read as $f=f_i=H_i=\mu\sigma u_i$, which is rotation dependent as the illumination $g=g_i$ in \eqref{diffusion} rotates along with the detector $\Gamma_i$. There is therefore no hope to reconstruct all sources $f_i$ from measurements of the form \eqref{mes1} (with $f$ replaced by $f_i$) unless $\Gamma_i=\Gamma$. All the sources $f_i\equiv H_i$ have to be anchored to a rotation-independent object, namely the domain of interest modeled by the optical parameters $(D,\sigma)$. This is the objective of the next section.

\subsection{Light propagation and boundary conditions}
\label{LPBC}
The second order elliptic PDE \eqref{diffusion} serves as a reasonable model for propagation of light in in highly scattering media such as biological tissues. 

We now aim to understand how the optical parameters $(D,\sigma)$ influence the ultrasound sources $H=\mu\sigma u$. Since $H\equiv H_i$ corresponds to rotating illuminations $g=g_i$ in \eqref{diffusion}, and rotating ultrasound measurements such as \eqref{mes1} do not allow full reconstructions of each $H_i$, we are also forced to understand such an influence locally, and in fact micro-locally. As in the derivation of \eqref{eq:vFIO}, the resulting pseudo-differential calculus requires enough smoothness for all the required Taylor-type expansions to make sense. We therefore assume that $(D,\sigma)$ are smooth, as we did for the sound speed $c(x)$. A laborious, standard, tracking of all relevant calculations shows that only finitely many terms are necessary in each Taylor expansion. As a consequence, all results hold for $(D,\sigma,c)$ of class $C^k$ for $k$ sufficiently large. We will present all results assuming $k=\infty$ both to simplify and stress that in practice, the difference between large $k$ and yet larger $k$ is somewhat immaterial.


Suppose that the object of interest occupies an open domain $\Omega \in \mathbb{R}^d, d=2,3$, where the boundary $\partial \Omega$ is $C^\infty$. 
Assume $D \in C^\infty(\mathbb{R}^d)$, and such that there is a positive constant $C_D$ and $C_D^{-1} \leq  D \leq C_D $, and that $\text{supp}(D-1) \subset \Omega$. For $ 0 \leq \sigma \in C^\infty_0(\mathbb{R}^d)$, assume there is some closed $\tilde{\Omega}$ such that $\Omega \subset \tilde{\Omega}$, and that $\text{supp}(\sigma)=\tilde{\Omega}$. \\
\newline 
It is well-known in the non-rotating setting that stable reconstruction of both $D$ and $\sigma$ requires internal functionals $H(x)$ from multiple different illuminations, and that appropriate illuminations exist \cite{bal2010inverse,bal2011multi}. Consider such a set of boundary conditions $\{g_j\}_{j=1}^N$. When the object is rotating and the detectors and illumination patterns are fixed, or equivalently, the detector and illumination patterns rotate, as for example in \cite{rotexp2,LNWEOA-SPIE-15}, we get a different set of $N$ illuminations for each rotation $i = 1,2,...M$, i.e., $\{g_{i,j}\}_{j=1}^N$. Now, for a fixed $i$, let $u_{i,j}$ be the solution to 
\begin{equation}
    \begin{split}
        -\nabla\cdot D \nabla u_{i,j} + \sigma u_{i,j} = 0  \quad \text{in } \Omega, \\
        u_{i,j} = g_{i,j} \quad \text{on } \partial \Omega, \quad j = 1,2,...,N. 
    \end{split}
    \label{diffusion2}
\end{equation}
If the $g_{i,j}$ are smooth, we have that $u_{i,j} \in C^\infty(\Omega)$ (\cite{evans10}, Theorem 3, Chapter 6.3). \\
We already mentioned that, regardless of the number of illuminations, one cannot reconstruct both $D$,$\sigma$ and the Gr\"{u}neisen coefficient $\mu$ \cite{bal2011multi}. To simplify notation, we assume $\mu $ known and set
\begin{equation*}
    H_{i,j}(x) = \sigma(x)u_{i,j}(x),   
\end{equation*}
where $u_{i,j}$ is the solution to \eqref{diffusion2}. Indexing the acoustic measurements accordingly, we have that 
\begin{equation}
    V_{i,j} = \chi_i \Lambda H_{i,j}, \quad \text{for } i = 1,2,...,M, \quad j = 1,2,...,N. 
    \label{measurement1}
\end{equation}
This provides a full description of our measurement setting. The parameter $N$ provides the diversity in boundary illuminations that is necessary to stably reconstruct the optical parameters in the second step of standard quantitative PAT \cite{bal2010inverse,bal2011multi}. The parameter $M$ indicates the number of rotations necessary to obtain measurements over all of $\Gamma$, and so roughly corresponds to the ratio between the size (volume of the hypersurface) of $\Gamma$ divided by that of the support of the rotating detector.

\subsection{Main results}
\label{MR}

We can now state the main result of the paper. This is, under some restrictive assumptions, an injectivity and stability result. It does not provide a reconstruction algorithm, merely the reassurance that enough information has been collected to uniquely and stably characterize the unknown optical coefficients. 

We therefore consider two pairs of admissible diffusion and absorption coefficients $(D, \sigma)$ and $(\bar{D},\bar{\sigma})$, satisfying the assumptions from Section \ref{LPBC}. Since the PDO and FIO calculus we use here is not meant to handle boundaries, we have to assume that the values of the coefficients agree in a neighborhood of $\partial \Omega$. For fixed $i,j$, let $u_{i,j}$ and $\bar{u}_{i,j}$ be solutions to \eqref{diffusion2} with illumination $g_{i,j}$ and optical parameters 
$(D, \sigma)$, $(\bar{D},\bar{\sigma})$, respectively. Denote $\delta H_{i,j} = \sigma u_{i,j} - \bar{\sigma} \bar{u}_{i,j}$, $\delta D = D - \bar{D}$ and $\delta \sigma = \sigma -\bar{\sigma}$.
We collect the measurements for each illumination and write their difference as
\begin{equation}
    \delta V_{i,j} = V_{i,j} - \overline{V}_{i,j} = \chi_i  \Lambda \delta H_{i,j}, \quad \text{for } i = 1,2,...,M, \quad j = 1,2,...,N. 
    \label{dvj}
\end{equation}

We use here the fact that the wave propagation step is linear: the measurements are linear in the source terms $H_i$. The full inverse problem, which maps $(D,\sigma)$ to such measurements, is however nonlinear. In spite of this, we will show in the next section that the terms $\delta H_{i,j}$ can be written as a functional that is nonlinear in $(D,\sigma,\bar D,\bar\sigma,u_{i,j},{\bar u}_{i,j})$ but linear in $(\delta D,\delta\sigma)$, essentially as a generalization of the fact that for any polynomial $p(x)$, we can find another polynomial $q(x,y)$ such that $p(x)-p(y)=q(x,y)(x-y)$.

Moreover, we will show in the next section that there exist open sets of illuminations $\{g_{i,j}\}$ such that the mapping from $(\delta D,\delta\sigma)$ to $\{\delta H_{i,j}\}$ may be described as an elliptic pseudo-differential operator (PDO) with a symbol that depends on $(D,\sigma,\bar D,\bar\sigma,u_{i,j},{\bar u}_{i,j})$. The first result along these lines was obtained in the non-rotating setting (with $M\equiv1$) in \cite{kuchment}.

Combining the elliptic FIO in \eqref{eq:vFIO} with the above elliptic PDO shows that the mapping from $(\delta D,\delta\sigma)$ to $\delta V_{j}$ is itself an elliptic FIO with symbol that depends on the $(D,\sigma,\bar D,\bar\sigma,u_{i,j},{\bar u}_{i,j},c(x))$, which must all be sufficiently smooth for the calculus to apply.

Such micro-local results provide optimal stability estimates as well as an injectivity result provided that a smoothing compact operator does not have eigenvalue one. This assumption is notoriously difficult to verify, unless we have recourse to a smallness assumption somewhere. In \cite{rotmes}, where the PAT problem with constant $D$ is considered, the smallness assumption was on the size of the absorption coefficient. Here, we make no assumption on the $O(1)$ size of the absorption and diffusion coefficients. Rather, we assume that the support of the domain where $\delta D = D - \bar{D}$ and $\delta \sigma = \sigma -\bar{\sigma}$ are unknown is itself sufficiently small. This will prove to be a sufficient assumption to obtain the following result. 



\begin{theorem}
\label{stability}
 If $\text{supp}(\sigma-\bar{\sigma})$ and $\text{supp}(D-\bar{D})$ are contained in sufficiently small ball $B_\varepsilon \subset \Omega$, there exist for each rotation $1 \leq i \leq M$ an open set of $2d$ illuminations $\{g_{i,j}\}_{j=1}^{2d}$ such that the following estimate holds for the corresponding measurements.
\begin{equation}
    \|\delta D \|_{L^2(\Omega)} + \|\delta \sigma \|_{H^1(\Omega)} \leq C\left( \sum_{i=1}^M \sum_{j=1}^{2d} \| \delta V_{i,j} \|_{H^1([0,T]\times \partial\Omega_M)}   \right).
    \label{ineq1}
\end{equation}
The constant $C$ depends on $D,\bar{D},\sigma, \bar{\sigma}, c, \Omega$ and the illuminations.  
\end{theorem}

Additional remarks on the set of necessary illuminations will be provided in the next section. 
We note that in dimension $d\geq3$, no boundary conditions $g$ guarantee the necessary ellipticity assumptions to obtain the optimal stability estimates given in the above theorem independently of the coefficients $(\sigma,D)$ \cite{ABD-ARMA-17,alberti}. 

\section{Construction of elliptic operator}
\label{sec:ell}
In lines with \cite{kuchment,rotmes}, we show that the problem mapping the unknown coefficients to the (also unknown) ultrasound sources $\delta H_{i,j}$ is described by an elliptic pseudo-differential operator for an appropriate choice of illuminations and for smooth coefficients and solutions of \eqref{diffusion}. 
\subsection{Parametrix for $\delta H_{i,j}$}
\label{linear}
Let the situation be as described in Sections \ref{LPBC}-\ref{MR}. Denote $\delta u_{i,j} = u_{i,j} - \bar{u}_{i,j}.$
Subtracting the equations \eqref{diffusion2} for $u_{i,j}$ and $\bar{u}_{i,j}$, we obtain an equation for $\delta u_{i,j}$ in terms of $\delta D$ and $\delta\sigma$: 

\begin{equation}
    \begin{split}
        -\nabla \cdot D \nabla \delta u_{i,j} + \sigma\delta u_{i,j} &= \nabla \cdot \delta D \nabla \bar{u}_{i,j} - \delta \sigma \bar{u}_{i,j} \quad \text{in } \Omega, \\
        \delta u_{i,j} &= 0 \quad \text{on } \partial \Omega.
    \end{split}
    \label{lin}
\end{equation}

We define the operators 
\begin{align*}
P_l(x,\partial) &= - \nabla \cdot D(x)\nabla + \sigma(x), \\    
P_r^{i,j}(x,\partial) &= \nabla \bar{u}_{i,j}(x) \cdot \nabla + \Delta\bar{u}_{i,j}(x),
\end{align*}
with symbols
\begin{align*}
p_l(x,\xi) &=  D(x)\xi^2 - i\xi \cdot \nabla D(x) + \sigma(x) \in S^2(\Omega,\mathbb{R}^n), \\ p_r^{i,j}(x,\xi) &= -i\nabla \bar{u}_{i,j}(x) \cdot \xi + \Delta\bar{u}_{i,j} \in S^1(\Omega,\mathbb{R}^n). 
\end{align*}
Recall that a function $p(x,\xi) \in C^\infty(\Omega \times \mathbb{R}^d)$ is said to be a symbol of class $S^m(\Omega \times \mathbb{R}^d)$, $m \in \mathbb{R}$, if for any multi-indices $\alpha,\beta \in \mathbb{N}^d$, there is a constant $C_{\alpha, \beta, \Omega, }$ such that 
\begin{equation*} 
|D_x^\alpha D_\xi^\beta p(x,\xi) | \leq C_{\alpha,\beta,\Omega}(1+|\xi|)^{m -|\beta|}, \quad \forall(x,\xi) \in \Omega \times \mathbb{R}^d.
\end{equation*}
In particular, the symbols $p_l$ and $p_r^{i,j}$ are homogeneous in $\xi$, and their principal symbols are $p_{l,0}(x,\xi) = D(x)\xi^2$ and $p_{r,0}^{i,j}(x,\xi) = -i\nabla \bar{u}_{i,j} \cdot \xi$. Since $D(x)$ is positive on $\Omega$, $P_l$ is elliptic and  $p_{l,0}(x,\xi) \sim \xi^2$. On the other hand, $p_{r,0}^{i,j}(x,\xi) $ is not elliptic at points $(x_0,\xi_0)$ such that $\nabla \bar{u}_{i,j}(x_0) \cdot \xi_0 = 0$. 
Let $Q$ be the parametrix (an inverse modulo smooth terms) of $P_l$,  with symbol $q$ and principal term $q_0$. Since $D(x),\sigma(x) > 0$ in $\Omega$, we have that $q(x,\xi) = \frac{1}{\xi^2+1} \text{ mod }S^{-3}(\Omega,\mathbb{R}^n)$, meaning that $q - \frac{1}{\xi^2+1} \in S^{-3}(\Omega,\mathbb{R}^d).$
Now we can solve \eqref{lin} for $\delta u_{i,j}$, modulo smooth terms: 
\begin{equation*}
    \delta u_{i,j} = QP_r^{i,j} \delta D - Q \bar{u}_{i,j}\delta\sigma. 
\end{equation*}
The symbol $r^{i,j}$ of the composition $QP_r^{i,j}$ has the asymptotic expansion 
$$ r^{i,j} \sim \sum_{\alpha \in \mathbb{N}_0^d} \frac{1}{\alpha!}\partial_\xi^\alpha q D_x^\alpha p_r^{i,j}, \quad \text{where } D_x^{\alpha}=(-i)^{|\alpha|}\partial_{x}^{\alpha},$$
and hence the principal part is 
\begin{equation*}
    r_0^{i,j}(x,\xi) = \frac{-i\nabla \bar{u}_{i,j}(x) \cdot \xi}{\xi^2+1} \text{ mod } S^{-2}(\Omega,\mathbb{R}^n). 
\end{equation*}

Next, consider $\delta H_{i,j} = H_{i,j} - \bar{H}_{i,j} = \sigma u_{i,j} - \bar{\sigma}\bar{u}_{i,j} $. Rearranging, we have
$$ \delta H_{i,j} = \sigma \delta u_{i,j} + \delta \sigma \bar{u}_{i,j}, $$
and hence we can write $\delta H_{i,j}$ as the result of a pseudo-differential operator acting on $\delta D$ and $\delta \sigma$, i.e,  
$$ \delta H_{i,j} = \sigma (QP_r^{i,j} \delta D - Q \bar{u}_{i,j} \delta \sigma) + \bar{u}_{i,j} \delta \sigma,$$
modulo smooth terms. We define  
\begin{equation}
\mathcal{H}^{i,j}_D = \sigma Q P_r^{i,j},  \quad  \quad \mathcal{H}^{i,j}_\sigma = \bar{u}_{i,j}(I - Q \sigma), \quad \text{and} \quad \mathcal{H}^{i,j} = [\mathcal{H}^{i,j}_D,\mathcal{H}^{i,j}_\sigma]. \label{Hs}
\end{equation}

In terms of \eqref{Hs} we have
\begin{equation*}
    \delta H_{i,j} = \mathcal{H}^{i,j}\begin{bmatrix}
    \delta D \\ \delta \sigma 
    \end{bmatrix} = \mathcal{H}^{i,j}_D \delta D + \mathcal{H}^{i,j}_\sigma \delta \sigma  + S_{i,j}\delta p, 
\end{equation*}
where $S_{i,j}$ is a linear operator with smooth Schwartz kernel
and $\delta p := (\delta D,\delta\sigma)$.
Let us finally note that the principal symbols of $\mathcal{H}^{i,j}_D$ and $\mathcal{H}^{i,j}_\sigma$ are 
$h^{i,j}_{D,0}(x,\xi) = r_0^{i,j}(x,\xi)$ and $h_{\sigma,0}^{i,j}(x,\xi) = \bar{u}_{i,j}(x)$, respectively. 
\subsection{Coupling with wave propagation}
\label{coupling}
From Section \ref{SP} and \eqref{measurement1}, the measurement $V_{i,j}$ is given by 
\begin{equation*}
    V_{i,j} = \chi_i (2 \pi)^{-d} \sum_{\tau = \pm}\int \int \mathrm{e}^{i\phi_\tau(t,x,\xi)-i\xi \cdot y
    }a_\tau(t,x,\xi)H_{i,j}(y) \mathrm{d}y\mathrm{d}\xi + \chi_i R H_{i,j}.
\end{equation*}
Letting $V_{i,j}$ and $\bar{V}_{i,j}$ be solutions to \eqref{wave2} with initial conditions $H_{i,j}$ and $\bar{H}_{i,j}$ respectively, we set $\delta V_{i,j} = V_{i,j} -\bar{V}_{i,j}$. Due to the linearity of \eqref{wave2}, we have 
\begin{align*}
    \delta V_{i,j} &= \chi_i (2 \pi)^{-d} \sum_{\tau = \pm}\int \int \mathrm{e}^{i\phi_\tau(t,x,\xi)-i\xi \cdot y
    }a_\tau(t,x,\xi)\delta H_{i,j}(y) \mathrm{d}y\mathrm{d}\xi + \chi_i R \delta H_{i,j}\\
    &= \chi_i (2 \pi)^{-d} \sum_{\tau = \pm}\int \int \mathrm{e}^{i\phi_\tau(t,x,\xi)-i\xi \cdot y
    }a_\tau(t,x,\xi)\left(\mathcal{H}^{i,j}_D \delta D + \mathcal{H}^{i,j}_\sigma \delta \sigma \right) \mathrm{d}y\mathrm{d}\xi + \chi_i R S_{i,j} \delta p. \nonumber
\end{align*}
The composition of an FIO with a PDO is well defined: from (\cite{FIO}, Theorem 4.2), the resulting operator is again a FIO with the same phase function and with a amplitude function with asymptotic expansion given by 
\begin{equation}
    c_\tau^{i,j}(t,x,y,\xi) \sim \sum_{\alpha \in \mathbb{N}^d_0} \frac{i^{-|\alpha|}}{\alpha!}\partial_y^\alpha
    \left(a_{\tau}(t,x,\xi) \partial_\xi^\alpha(h_D^{i,j}(y,\xi) + h_\sigma^{i,j}(y,\xi))\right), \quad \tau = \pm.
    \label{amplitude}
\end{equation}
From \cite{uhlmann}, we know that $a_\tau$ is a zeroth order amplitude with $a_{\tau,0}=\frac{1}{2}$. Hence the principal term of \eqref{amplitude}
is $\frac{1}{2}(h_{D,0}^{i,j}(y,\xi) + h_{\sigma,0}^{i,j}(y,\xi)) \text{ mod } S^{-2}(\Omega \times \mathbb{R}^d), \tau = \pm.$ \\
The forward map now takes the form
\begin{equation}
    \delta V_{i,j} = \chi_i \Lambda \mathcal{H}^{i,j} \begin{bmatrix}
    \delta D \\ \delta \sigma
    \end{bmatrix} =\chi_i (2 \pi)^{-d} \sum_{\tau = \pm}\int \int \mathrm{e}^{i\phi_\tau(t,x,\xi)-i\xi \cdot y
    }c_\tau^{i,j}(t,x,y,\xi)\begin{bmatrix}
    \delta D \\ \delta \sigma
    \end{bmatrix} \mathrm{d}y\mathrm{d}\xi, 
    \label{FIOcomp}
\end{equation}
modulo a smooth term $Q_{i,j}\delta p$ involving a linear operator $Q_{i,j}$ with smooth Schwartz kernel.
\subsection{Time-reversal}
To bring us back from boundary measurements to objects (such as $\delta p$) that are defined on the spatial domain $\Omega$, we  apply the time-reversal inversion to the composed operator in \eqref{FIOcomp}. The time reversal operator is an approximate inverse to $\Lambda$, which we denote by $A$ \cite{uhlmann,timereversal,timereversal2}. For given data $h \in  H^1([0,T]\times \partial \Omega_M)$, let $w$ be the solution of 
\begin{equation*}
    \begin{split}
        &(\partial_t^2 - c(x)^2\Delta) w(t,x) = 0, \quad \text{in }(0,\infty)\times\mathbb{R}^n,\\
        &w(t,x) = h, \quad (t,x) \in [0,T]\times \partial \Omega_M, \\
        &w(T,x) = \phi, \\
        &\partial_t w(T,x) = 0, 
        \end{split}
\end{equation*}
where $\phi$ solves $\Delta \phi = 0, \phi|_{\partial \Omega_M} = h(T,\cdot)$. 
Then $A h = w\big|_{t=0,x\in \Omega}$. For $f \in C^\infty_0(\Omega)$, we have that $f - A\Lambda f = Kf$, where $K$ is a compact operator with $\|K\|<1$. See \cite{uhlmann} for details.  \\
For a fixed measurement position $i$ and illuminations $g_{i,j}$, we now consider the system 
\begin{equation}
    A\begin{bmatrix}
    \delta V_{i,1}  \\ \vdots \\ \delta V_{i,N} 
    \end{bmatrix} = A\begin{bmatrix}
    \chi_i \Lambda \mathcal{H}^{i,1} \\ \vdots \\ \chi_i\Lambda \mathcal{H}^{i,N}
    \end{bmatrix}\begin{bmatrix}
    \delta D \\ \delta \sigma 
    \end{bmatrix} +AQ_i\delta p, 
    \label{FIOsys}
\end{equation}
where we stacked the $\delta V_{i,j}$ and applied the time reversal operator to each side. We now define 
$$ \kappa_{i}  = A\begin{bmatrix}
    \chi_i \Lambda \mathcal{H}^{i,1} \\ \vdots \\ \chi_i\Lambda \mathcal{H}^{i,N}
    \end{bmatrix}.$$ For a more compact notation in what follows, we write 
$$ \tilde{\chi}_i(x,\xi)=\frac{1}{2}\big(\chi_i(\tau_+(x,\xi),\gamma_{x,\xi}(\tau_+(x,\xi))) + \chi_i(\tau_-(x,\xi),\gamma_{x,\xi}(\tau_-(x,\xi)))\big).$$ We will need the following lemma. 
\begin{lemma}
The operator $\kappa_{i}$ is a pseudo-differential operator with principal symbol $\kappa_{i}^0$ given by 
\begin{equation}
\kappa_{i}^0(x,\xi)= 
\begin{bmatrix}
\tilde{\chi}_i(x,\xi)\frac{-i\nabla \bar{u}_{i,1}(x) \cdot \xi}{\xi^2+1} & \quad \quad  \tilde{\chi}_i(x,\xi) \bar{u}_{i,1}(x) \\
\vdots & \quad \vdots \\
\tilde{\chi}_i(x,\xi)\frac{-i\nabla \bar{u}_{i,N}(x) \cdot \xi}{\xi^2+1} &\quad  \tilde{\chi}_i(x,\xi) \bar{u}_{i,N}(x)
\end{bmatrix}.
\label{compsym}
\end{equation}
\label{symbol}
\end{lemma}

\begin{proof}
It follows the linearity of $A$ and Theorem 3 in \cite{uhlmann} that 
$ A\left( \chi_i \Lambda\right) $ is a pseudo-differential operator of order zero with principal symbol $ \tilde{\chi}_i(x,\xi)$.
Since $A(\Lambda \mathcal{H}^{i,j}) =(A\Lambda) \mathcal{H}^{i,j}$, we have that 
$A\left( \chi_i \Lambda \mathcal{H}^{i,j}\right)$ is again a pseudo-differential operator with symbol $\tilde{\chi}_i(x,\xi)(h^{i,j}_D(x,\xi) + h^{i,j}_\sigma(x,\xi))$. Considering the expressions for 
$h^{i,j}_D(x,\xi)$ and $h^{i,j}_\sigma(x,\xi)$ found in Section \ref{linear}, equation \eqref{compsym} follows. 
\end{proof}    

By applying the adjoint of $\kappa_i$, $\kappa_i^*$ to \eqref{FIOsys} and summing over $i$ we get the system
\begin{equation}
\delta \mathcal{V}
    = \sum_{i=1}^M \kappa_i^*\kappa_i\begin{bmatrix}
    \delta D \\ \delta \sigma 
    \end{bmatrix} +P\delta p,  
    \label{FIOsys2}
\end{equation}
where 
\begin{equation}
    \delta \mathcal{V} = \sum_{i=1}^M \kappa_i^*A\begin{bmatrix}
    \delta V_{i,1}  \\ \vdots \\ \delta V_{i,N} 
    \end{bmatrix} \quad \text{and } \quad P = \sum_{i=1}^M \kappa_i^*AQ_i.
    \label{dV}
\end{equation} 
Now $\delta \mathcal{V}$ contains all the measured data, and the linear operator $P$ has a smooth Schwartz kernel.
\newline 
    
We thus have replaced our inversion problem by the analysis of equation \eqref{FIOsys2}. 

\subsection{Analysis of the system of pseudo-differential operators}
We now investigate the system 
\begin{equation}
      \sum_{i=1}^M \kappa_i^*\kappa_i\begin{bmatrix}
    \delta D \\ \delta \sigma 
    \end{bmatrix} +P\delta p =  \delta \mathcal{V}.
    \label{osys}
\end{equation}
Such a system is called elliptic when the principal symbol $\sum_{i=1}^M (\kappa_i^0)^*\kappa_i^0$  of the operator $\sum_{i=1}^M \kappa_i^*\kappa_i$ is full rank for all $(x,\xi)\in T^*\Omega $, \cite{hormander,agranovich}. We show that there exist suitable sets of illuminations such that this is the case. \\
\newline 
Since each of the terms in $\sum_{i=1}^M (\kappa_i^0)^*\kappa_i^0$ is positive semi-definite, the sum will also be positive semi-definite. Hence it suffices to show that for each $(x,\xi) \in T^*\Omega$, at least one term is positive definite, as this will then guarantee the positive definiteness and full rank property of the sum. 
Since $(\kappa_i^0)^*\kappa_i^0$ is positive when $\kappa_i^0$ is full rank, we aim to find illuminations such that for every $(x,\xi) \in T^*\Omega$, there is some $\kappa_i^0$ with full rank. 
Recall that

\begin{equation}
\kappa_{i}^0(x,\xi)= 
\begin{bmatrix}
\tilde{\chi}_i(x,\xi)\frac{-i\nabla \bar{u}_{i,1}(x) \cdot \xi}{\xi^2+1} & \quad \quad  \tilde{\chi}_i(x,\xi) \bar{u}_{i,1}(x) \\
\vdots & \quad \vdots \\
\tilde{\chi}_i(x,\xi)\frac{-i\nabla \bar{u}_{i,N}(x) \cdot \xi}{\xi^2+1} &\quad  \tilde{\chi}_i(x,\xi) \bar{u}_{i,N}(x)
\end{bmatrix}.
\label{compsym2}
\end{equation}
\newline
From the assumptions on the non-trapping wave speed and sufficiently long measurement time, we know that for each  $(x,\xi) \in T^*\Omega$, at least one $\tilde{\chi}_i(x,\xi) > 0$. For $(x,\xi) \in T^*\Omega$ and such a rotation index $i=i(x,\xi)$, we need to show that $\kappa_{i}^0(x,\xi)$ is full rank.
\newline 
We denote the determinant of $2\times 2$ matrix consisting of rows $m$ and $n$ of \eqref{compsym2} by 
$q_{m,n}(x,\xi)$. Then 
\begin{equation}
    q_{m,n}(x,\xi) = \tilde{\chi}_i^2(x,\xi)\frac{i}{1 + \xi^2}\xi \cdot (\nabla \bar{u}_{i,m} \bar{u}_{i,n} - \nabla \bar{u}_{i,n} \bar{u}_{i,m}).
    \label{q}
\end{equation}
For $\kappa_i^0(x,\xi)$ to be full-rank, we thus require that there is at least one $q_{m,n}(x,\xi) \neq 0$ for every $(x,\xi) \in \Omega \times S^{d-1}$.  Here $S^{d-1}$ denotes the unit sphere in $\mathbb{R}^d$, and replaces $\mathbb{R}^d\setminus\{0\}$ since only the direction of $\xi$ is of importance. 
\newline 
Hence we seek to find illuminations $\{g_{i,j} \}$ such that for some indices $(m,n)$ it holds that  
\begin{equation}
    \tilde{q}_{m,n}(x,\xi) = \xi \cdot (\nabla \bar{u}_{i,m} \bar{u}_{i,n} - \nabla \bar{u}_{i,n} \bar{u}_{i,m}) \neq 0, \quad (x,\xi) \in \Omega \times S^{d-1}, \quad \forall 1 \leq i \leq M.
    \label{qt}
\end{equation}

Expressions similar to \eqref{qt} appears in the literature on PAT, e.g., in \cite{bal2010inverse,bal2011multi,kuchment,alberti}. \\
\newline 
For \eqref{qt} to hold, we need that the vector fields  
\begin{equation*}
    v_{m,n}(x) = (\nabla \bar{u}_{i,m} \bar{u}_{i,n} - \nabla \bar{u}_{i,n} \bar{u}_{i,m})(x)
\end{equation*}
constitute a basis for $\mathbb{R}^d$ for each $x \in \Omega$. In that way, a direction $\xi_0 \in S^{d-1}$ can never be orthogonal to all $v_{m,n}$.  It is therefore clear that we must have at least $d+1$ different illuminations $\{g_{i,j}\}_{j=1}^{d+1}$. We will use complex geometric optics (CGO) solutions to show that there exists a set of $2d$ illuminations such that the vector fields $v_{m,n}$ does indeed form a basis for $\mathbb{R}^d$ for each $x \in \Omega$,  and also show that under certain restrictions on the optical coefficients, the same boundary conditions will work for all rotations.
We briefly introduce the CGO solutions first.
\newline 
\newline 
By the Liouville transformation $v = \sqrt{D}u$, equation \eqref{diffusion2} is written as
\begin{equation}
    \begin{split}
    -\Delta v_{i,j} + q v_{i,j} &= 0, \quad \text{in } \Omega, \\
    v_{i,j} &= \tilde{g}_{i,j}, \quad \text{on } \partial\Omega, 
    \end{split}
    \label{schrodinger}
\end{equation}
where $q(x) = \Delta \sqrt{D}/\sqrt{D} + \sigma/D$. The CGO solutions are special solutions to \eqref{schrodinger} that are perturbations of complex plane waves $\mathrm{e}^{\rho \cdot x}$ of the form
$$ v_\rho (x) = \mathrm{e}^{\rho \cdot x}(1 + \psi_\rho (x)),$$
where $ \rho \in \mathbb{C}^d$ and $\rho \cdot \rho = 0$. 
For $|\rho|$ large enough, such solutions exist, and from (\cite{bal2010inverse}, Corollary 3.2), it is known that the perturbation term $\psi_\rho$ satisfies the bound 
\begin{equation}
    \|\psi_\rho \|_{H^{s}(\Omega)} \leq C(\Omega) \frac{\|q \|_{H^{s}(\Omega)} }{|\rho|},
    \label{qb}
\end{equation}
where $s > d/2 + k$ for $k \geq 1$. By choosing the parameter $\rho$ in a certain way, we can tailor these solutions to achieve solutions with the right properties. 
We present the result for the situation when $d=3$.
\begin{theorem}
\label{cgo}
For any rotation $i$ there is a open set of illuminations $\{g_{i,j}\}_{j=1}^6$ such that for each $x\in \Omega$ 
the vector fields $$ v_{n,m}(x) = (\nabla u_{i,m} u_{i,n} - \nabla u_{i,n} u_{i,m})(x), \quad 1 \leq n, m \leq 6,$$
form a basis for $\mathbb{R}^3$. \\
\newline
In addition, if $D$ is constant in a neighborhood of the boundary, there exists a $\delta$ such that if $\|q\|_{H^{d/2 + k + \varepsilon}} \leq \delta$, 
then the same set of illuminations work for all rotations.

\end{theorem}
By open set, we mean  $\{\tilde{g}_{i,j}\}$ sufficiently close to $\{g_{i,j}\}$ in any topology of sufficiently smooth functions such as $C^2(\bar\Omega)$. 
The proof combines ideas from \cite{alberti} and \cite{bal2011densities}. 
\begin{proof}
We choose CGO-parameters $\rho_1 = t (\mathbf{e}_2 + i \mathbf{e}_1)$, $\rho_2 = t(\mathbf{e}_3 + i \mathbf{e}_1)$ and $\rho_3 = t (-\mathbf{e}_2 + i \mathbf{e}_1)$, and choose the corresponding (real-valued) solutions to be 
\begin{align*}
u_{i,1} = \text{Re}\{D^{-1/2}v_{\rho_1}(x)\}, \quad  u_{i,2} = \text{Im}\{D^{-1/2}v_{\rho_1}(x)\}, \quad u_{i,3} = \text{Re}\{D^{-1/2}v_{\rho_2}(x)\}, \\
u_{i,4} = \text{Im}\{D^{-1/2}v_{\rho_2}(x)\}, \quad  u_{i,5} = \text{Re}\{D^{-1/2}v_{\rho_3}(x)\}, \quad u_{i,6} = \text{Im}\{D^{-1/2}v_{\rho_3}(x)\}.
\end{align*}
After some algebra, it follows that 
\begin{align*}
    |\text{det}[ \nabla u_{i,1} \quad \nabla u_{i,2}  \quad \nabla u_{i,3}] (x)| &= Ct^3\mathrm{e}^{t(2\mathbf{e}_2 +\mathbf{e}_3)\cdot x}|(1+O(t^{-1}))\cos(t\mathbf{e}_1\cdot x)|, \\
    |\text{det}[ \nabla u_{i,1} \quad  \nabla u_{i,2} \quad \nabla u_{i,4}] (x)|  
    &= Ct^3\mathrm{e}^{t(2\mathbf{e}_2 +\mathbf{e}_3)\cdot x}|(1+O(t^{-1}))\sin(t\mathbf{e}_1\cdot x)|,\\
    u_{i,5}(x) &= \mathrm{e}^{-t\mathbf{e}_3 \cdot x}\cos(t \mathbf{e}_1 \cdot x)(1 + O(t^{-1})), \\
    u_{i,6}(x) &= \mathrm{e}^{-t\mathbf{e}_3 \cdot x}\sin(t \mathbf{e}_1 \cdot x)(1 + O(t^{-1})),
    \label{cgo}
\end{align*} 
where $C>0$ is independent of $t$. We take the smallest $t$ so large that $|O(t^{-1})| \leq 1/2 $ in all expressions above. Since the zero sets of the determinants are disjoint, it follows that there are complementary sets $\Omega_1,\Omega_2$ such that $\Omega = \Omega_1 \cup \Omega_2$ and that $\{\nabla u_{i,1}, \nabla u_{i,2}, \nabla u_{i,3} \}$ is a basis for $\mathbb{R}^3$ on $\Omega_1$, and $\{\nabla u_{i,1}, \nabla u_{i,2}, \nabla u_{i,4} \}$ is a basis for $\mathbb{R}^3$ on $\Omega_2$ (see \cite{bal2011densities} for details on the construction). Correspondingly, $u_{i,5}(x) \neq 0 $ for $x \in \Omega_1$ and $u_{i,6}(x) \neq 0 $ for $x \in \Omega_2$. 
As a consequence, the vector fields 
$$ v_{5,1},v_{5,2},v_{5,3},v_{6,1},v_{6,2},v_{6,4}$$
form a basis for $\mathbb{R}^3$ for every $x \in \Omega$. Hence we can choose illuminations such that $g_{i,j} = u_{i,j}|_{\partial \Omega}$. Further, by continuity of the mapping $g \in C^2(\overline{\Omega}) \mapsto u \in C^1(\overline{\Omega})$, it suffices to take $g_{i,j}$  close to  $u_{i,j}|_{\partial \Omega}$, i.e., 
$\|g_{i,j} - u_{i,j}\|_{C^2(\overline{\Omega})} \leq \varepsilon$, for some $\varepsilon > 0$ sufficiently small \cite{alberti}. \\
\newline 
Now, let $R \in \mathbb{R}^{3\times3}$ be a rotation matrix, and set $\tilde{q}(x) = q(Rx)$. Then equation \eqref{schrodinger} with $q$ replaced by $\tilde{q}$ corresponds to a rotation $i$ of the object or the illumination patterns. Let 
\begin{equation}
    \tilde{v}_{\rho_1}(x) = \mathrm{e}^{\rho_i\cdot x}(1 + \tilde{\psi}_{\rho_1}(x)), \quad \tilde{u}_{k,1} =\text{Re}\{ D^{-1/2}\tilde{v}_{\rho_1}(x)\},
\end{equation}
where $\tilde{\psi}_{\rho_i}$ is the perturbation term with respect to $\tilde{q}$. 
We assume that $\|g_{i,1} - u_{i,j}|_{\partial \Omega}\|_{C^2(\overline{\Omega})} \leq \varepsilon/2$. Then 
\begin{align*}
\|g_{i,1} - \tilde{u}_{k,1}\|_{C^2(\overline{\Omega})} &= 
\|g_{i,1} - u_{k,1}+ u_{k,1} -\tilde{u}\|_{C^2(\overline{\Omega})}\\ 
&\leq \varepsilon/2 + \| u_{k,1} -\tilde{u}_{k,1} \|_{C^2(\overline{\Omega})} \\
&= \varepsilon/2 + \|\text{Re}\{\mathrm{e}^{\rho_1 \cdot x}(\psi_{\rho_1} - \tilde{\psi}_{\rho_1})\  \|_{C^2(\overline{\Omega})}\\
&\leq \varepsilon/2 + \| \text{Re}\{\mathrm{e}^{\rho_1 \cdot x} \}\|_{C^2(\overline{\Omega})}\frac{2}{|\rho_1|}\|q \|_{H^{s}},
\end{align*}
where the last inequality follows by Sobolev embedding and the bound in \eqref{qb}, setting $s > 3/2 + 2$, and the fact that $\|q\|_{H^s} = \|\tilde{q}\|_{H^s}$.
The result now follows by requiring 
$$ \|q\|_{H^s} \leq \delta = \frac{\varepsilon |\rho_1|}{4\|\text{Re}\{\mathrm{e}^{\rho_1 \cdot x}\} \|_{C^2(\overline{\Omega})}}.$$
The same argument then works for all illuminations.

\end{proof}

The drawback with using CGO-solutions is that the results are not constructive, since they rely on the unknown parameters. Numerical simulations \cite{bal2011multi} show that most illuminations will in fact work, but we know that in dimension $d\geq3$, we cannot find illuminations that provide ellipticity conditions independently of the coefficients \cite{ABD-ARMA-17}.

\subsection{Stability and uniqueness}
Written out, the principal part of measurement object $\delta \mathcal{V}$ in \eqref{dV} is of the form
\begin{equation*}
    \delta \mathcal{V} = \begin{bmatrix}
    \delta \mathcal{V}_1 \\ \delta \mathcal{V}_2 
    \end{bmatrix} 
\end{equation*}
with components 
\begin{equation*}
    \delta \mathcal{V}_1 =\sum_{i=1}^M \sum_{j=1}^N \tilde{\chi}_i\frac{i\nabla \bar{u}_{i,j}(x) \cdot \xi}{\xi^2+1}A\delta V_{i,j} \quad \text{and } \quad \delta \mathcal{V}_2 =\sum_{i=1}^M \sum_{j=1}^N \tilde{\chi}_i\bar{u}_{i,j}A\delta V_{i,j}.
\end{equation*}
Note that for $\delta \mathcal{V}_1$ and $\delta \mathcal{V}_2$, we have by standard mapping properties of PDOs on Sobolev spaces that 
\begin{align}
    \begin{split}
    \|\delta \mathcal{V}_1 \|_{H^s} &= \left\|\sum_{i=1}^M \sum_{j=1}^N \tilde{\chi}_i\frac{i\nabla \bar{u}_{i,j}(x) \cdot \xi}{\xi^2+1}A\delta V_{i,j} \right\|_{H^s} \leq  C_s\sum_{i=1}^M \sum_{j=1}^N\left\| A\delta V_{i,j} \right\|_{H^{s-1}},\\
    \|\delta \mathcal{V}_2 \|_{H^s} &= \left\|\sum_{i=1}^M \sum_{j=1}^N \tilde{\chi}\bar{u}_{i,j}A\delta V_{i,j} \right\|_{H^s} \leq \tilde{C}_s\sum_{i=1}^M \sum_{j=1}^N\left\| A\delta V_{i,j} \right\|_{H^{s}},
    \end{split}
    \label{sobo}
\end{align}
for $s\in \mathbb{R}$ and constants $C_s,\tilde{C}_s>0$. \\
\newline 
We now associate to the system \eqref{FIOsys2} two sets of integers  $s = (s_1,s_2)$ and $t = (t_1,t_2)$ such that for each entry of the $2\times 2$ symbol  $\sum_{i=1}^M(\kappa_i)^*\kappa_i$ we have that $\left(\sum_{i=1}^M(\kappa_i)^*\kappa_i(\kappa^0_i)^*\kappa^0_i\right)_{(i,j)} \in S^{s_i - t_j}(\Omega,\mathbb{R}^n)$. Since for each $i$
\begin{equation*}
    (\kappa^0_i)^*\kappa^0_i (x,\xi) = 
    \tilde{\chi}_i^2(x,\xi)\sum_{j=1}^N\begin{bmatrix} \frac{(\nabla \bar{u}_{i,j}(x) \cdot \xi)^2}{(\xi^2+1)^2} & \quad 
   \frac{i\nabla \bar{u}_{i,j}(x) \cdot \xi \bar{u}_{i,j}}{\xi^2+1} \\
    \frac{-i\nabla \bar{u}_{i,j}(x) \cdot \xi \bar{u}_{i,j}}{\xi^2+1} & \quad  \bar{u}_{i,j}^2 
    \end{bmatrix},
\end{equation*}
we can choose and $s = (-1,0)$ and $t = (1,0)$ so that the above entry $(k,l)$ is of order $s_k-t_l$.
\newline
The operator $(\kappa^0_i)^*\kappa^0_i$ is said to be elliptic of type $(s,t)$ in the Douglis-Nirenberg sense if it is invertible for every $(x,\xi) \in T^* \Omega$, see \cite{hormander,nirenberg}. When $(\kappa^0_i)^*\kappa^0_i$ is elliptic, it has a parametrix $Q(x,\xi)$, an inverse modulo smooth terms, and applying $Q$ to \eqref{osys} yields the estimate (since $QP$ is smooth as well)
\begin{equation}
    \|\delta D \|_{L^2(\Omega)} + \|\delta \sigma \|_{H^1(\Omega)} \leq C\left(\| \delta \mathcal{V}_1 \|_{H^2(\Omega)} + \| \delta \mathcal{V}_2 \|_{H^1(\Omega)} + \|\delta D \|_{H^{s'}(\Omega)} + \|\delta \sigma \|_{H^{s'}(\Omega)} \right),
    \label{elreg}
\end{equation}
for any $s' \in \mathbb{R}$ (\cite{hormander}, Lemma 1.0.2', or \cite{kuchment}). The constant $C$ depends on $D,\bar{D},\sigma, \bar{\sigma}, \Omega$, on $s'$ and the illuminations.
We can now complete the proof of Theorem \ref{stability}.
\begin{proof}
That boundary conditions such that $(\kappa^0_i)^*\kappa^0_i$ is elliptic of the $(s,t)$ type exist follows from Theorem \ref{cgo}, and hence the estimate \eqref{elreg} holds.
Let $s'=-1$ in \eqref{elreg}. For $u \in  L^2(\Omega) \subset H^{-1}(\Omega)$, $\text{supp } u \subset B_\varepsilon$, we have
\begin{align*}
    \|u\|_{H^{-1}(\Omega)} =& \sup_{\|v\|_{H^1_0(\Omega)} \leq 1} \bigg|\int_{B_\varepsilon} u v \mathrm{d}x\bigg| 
    = \sup_{\|v\|_{H^1_0(B_\varepsilon)} \leq 1} \bigg|\int_{B_\varepsilon} u v \mathrm{d}x \bigg|\\
    \leq &  \sup_{\|v\|_{H^1_0(B_\varepsilon)} \leq 1 }  \|u\|_{L^2(B_\varepsilon)} \|v\|_{L^2(B_\varepsilon)} 
    \leq  \sup_{\|v\|_{H^1_0(B_\varepsilon)} \leq 1 } \|u\|_{L^2(B_\varepsilon)} \tilde{C} \varepsilon\|Dv\|_{L^2(B_\varepsilon)} \\
    \leq & \sup_{\|v\|_{H^1_0(B_\varepsilon)} \leq 1 }\|u\|_{L^2(B_\varepsilon)} \tilde{C} \varepsilon\|v\|_{H^1_0(B_\varepsilon)} \leq \tilde{C} \varepsilon\|u\|_{L^2(B_\varepsilon)}, 
\end{align*}
where we use the Poincar\'e inequality to get the $\varepsilon$-dependence; see \cite{evans10}. A priori, $u \in H_0^1(B_\varepsilon)$, and we have that
\begin{equation} \|u\|_{H^{-1}(\Omega)} \leq \tilde{C} \varepsilon \|u\|_{L^2(\Omega)}, \qquad \|u\|_{H^{-1}(\Omega)} \leq \tilde{C} \varepsilon \|u\|_{H^1(\Omega)}.
\label{uest}
\end{equation}
Using the estimates in \eqref{sobo} we get that that for some constant $C_v > 0$
\begin{equation}
    \|\delta \mathcal{V}_1 \|_{H^2(\Omega)} + \| \delta \mathcal{V}_2 \|_{H^1(\Omega)} \leq C_v\sum_{i=1}^M \sum_{j=1}^N \| A \delta V_{i,j} \|_{H^1(\Omega)}.
\end{equation}
Last, the time-reversal operator $A$ is bounded (\cite{lions}, Theorem 2.4 with $\theta = 1$), and we have $\|A\delta V_{i,j} \|_{H^1(\Omega)} \leq C_A\|\delta V_{i,j} \|_{H^1([0,T]\times \partial \Omega_M)}.$
Under the assumption that $\text{supp } \delta D, \text{supp } \delta \sigma \subset B_\varepsilon$, we make use of the inequalities \eqref{uest} and rearrange \eqref{elreg} to get
\begin{equation}
    (1 - C\tilde{C}\varepsilon)(\|\delta D \|_{L^2(\Omega)} + \|\delta \sigma \|_{H^1(\Omega)} )\leq C_A C_v \sum_{i=1}^M\sum_{j=1}^N \|\delta V_{i,j} \|_{H^1([0,T]\times \partial \Omega_M)}.
    \label{elreg2}
\end{equation}
We can then choose $\varepsilon$ such that $(1-C\tilde{C}\varepsilon) >0$. Recalling Theorem \ref{cgo}, we have $N = 2d$ and we get \eqref{ineq1}. 
\end{proof}


\section{Remarks and Conclusions}
\label{sec:conclu}

In a sufficiently idealized setting, where the ultrasound propagation may be modeled by a wave equation with reasonably well known sound speed (and nagging attenuation effects may be neglected \cite{kowar2012attenuation}), quantitative PAT displays favorable mathematical properties, as a composition of two reasonably well-posed inverse problems. This holds in the setting where (sufficiently) full ultrasound measurements are available, which is not always practical \cite{rotexp2,rotexp}.

We show in this paper that PAT enjoys favorable stability properties as well when a similar amount of measurement is collected for optical illuminations and ultrasound detectors that are allowed to rotate during the acquisition procedure so that different parts of the ultrasound measurements are generated by different optical illuminations. 

The absence of complete ultrasound measurement for each optical illumination renders the inversion of a single wave equation ill-posed. Only when measurements from all optical sources are accounted for can one expect stable reconstructions.  In such settings, where only local pieces of information are available for each illumination, it is difficult to envision a direct global inversion procedure. Rather, it is the ideal setting to apply micro-local methods, whose roles are precisely to propagate (phase-space-) local information through wave or elliptic equations. 

Writing the difference of sufficiently smooth nonlinear functionals ${\cal F}(u)-{\cal F}(v)$ as a general functional ${\cal G}(u,v,u-v)$ that is linear in its last component, a procedure that applies in a great variety of contexts, we can write the difference of measurements as a (standard) elliptic Fourier integral operator applied to $u-v$ with a symbol that depends on $(u,v)$. This imposes that $(u,v)$, for us here the sound speed and the optical coefficients as well as the solutions to the diffusion equation, be sufficiently smooth. This procedure has been applied for the second step of many hybrid inverse problems in \cite{kuchment}. The inversion procedure then necessarily provides conditional stability, that is to say stable reconstructions up to the possibly non-trivial kernel of a compact operator. In this paper, we chose to make a smallness assumption on the support of the coefficients of interest to show that the kernel of said operator was necessarily trivial, a trick that is certainly not new \cite{H-III-SP-94}.

These are our two main somewhat unnecessary condition: very large smoothness of coefficients and smallness of support.

Such strong smoothness assumptions are not necessary in standard PAT \cite{bal2011multi} (although no known results hold for arbitrary coefficients; for instance $D$ and $\sigma$ arbitrary measurable functions that are bounded between $1$ and $2$, say). They can also be avoided for more general hybrid inverse problems by writing the inverse problem as a coupled system of nonlinear partial differential equations for $(D,\sigma,u_j)$, where the PDO calculus of the second step of PAT is replaced by a better behaving potential theory \cite{B-CM-14}. In that setting, one can also sometimes apply a unique continuation principle that allows one to obtain an injectivity result independently of the size of the support of the unknown coefficients.

However, in our setting of rotating measurements, it seems unclear how the micro-local pursuit of propagation of singularities can be totally avoided. Within that context, unnecessary smoothness assumptions seem necessary (realizing the inherent contradiction). As we mentioned in the introduction, all the smoothness we need is for $s'$ in  \eqref{elreg} to be strictly negative. This can be achieved by Taylor expansions involving finitely many terms in the definition of PDO and FIO operators as well as in their composition. How many (a sufficiently large finite number) dictates how much smoothness our coefficients need to verify. 

These technical constraints notwithstanding, we expect standard quantitative PAT and QPAT in a rotating setting to display very similar resolution capabilities.

\section*{Acknowledgment} The work of GB was partially funded by the US National Science Foundation and the US Office of Naval Research.



\bibliographystyle{plain}
\bibliography{bib}

\end{document}